\documentclass{amsart}
\usepackage[english]{babel}
\usepackage{amsmath}
\usepackage{amsfonts}
\usepackage{mathrsfs}
\usepackage{amssymb}
\usepackage{dsfont}
\usepackage{amsthm}
\usepackage{hyperref,polynom}
\usepackage{xfrac}
\usepackage{defs}
\usepackage{fonts}
\usepackage{cite}
\usepackage{wrapfig} 
\usepackage[export]{adjustbox} 
\usepackage{todonotes}
\usepackage{mathtools} 
\usepackage{graphicx} 
\usepackage{array} 
\usepackage{enumerate} 
\usepackage{hyperref} 

\newtheorem{theorem}{Theorem}[section]
\newtheorem{proposition}[theorem]{Proposition}
\theoremstyle{definition}

\theoremstyle{plain}
\newtheorem{lemma}[theorem]{Lemma}

\theoremstyle{remark}
\newtheorem{remark}[theorem]{Remark}

\title[Strong solvability of nonlocal systems]  
{The solvability of a strongly-coupled nonlocal system of equations 
}
\author
{Tadele Mengesha, James M. Scott}

\makeatletter
\def\namedlabel#1#2{\begingroup
    #2%
    \def\@currentlabel{#2}%
    \phantomsection\label{#1}\endgroup
}
\makeatother

\newcommand{\e}{\mathrm{e}}

\newcommand{\intdm}[3]{ \int_{#1} #2 \, \mathrm{d}#3}
\newcommand{\iintdm}[5]{\int_{#1} \int_{#2} #3 \, \mathrm{d}#4 \, \mathrm{d}#5}

\newcommand{\intdmt}[4]{\displaystyle \int_{#1}^{#2} #3 \, \mathrm{d}#4}

\newcommand{\diffqbunorm}{\left| (\bu(\bx)-\bu(\by)) \cdot \frac{\bx-\by}{|\bx-\by|}\right|}

\newcommand{\shapetensorby}{\left( \frac{\by \otimes \by}{|\by|^2} \right)}
\newcommand{\shapetensorbh}{\left( \frac{\bh \otimes \bh}{|\bh|^2} \right)}
\newcommand{\shapetensorbfxi}{\left( \frac{\bfxi \otimes \bfxi}{|\bfxi|^2} \right)}

\newcommand{\DotLap}{(-\mathring{\Delta})}

\newcommand{\bfxi}{\boldsymbol{\xi}}
\newcommand{\bfeta}{\boldsymbol{\eta}}

\begin{document}

\maketitle


\numberwithin{equation}{section}


\begin{abstract}
We prove existence and uniqueness of strong (pointwise) solutions to a linear nonlocal strongly coupled hyperbolic system of equations posed on all of Euclidean space. The system of equations comes from a linearization of a nonlocal model of elasticity in solid mechanics. It is a nonlocal analogue of the Navier-Lam\'e system of classical elasticity. We use a well-known semigroup technique that hinges on the strong solvability of the corresponding steady-state elliptic system. The leading operator is an integro-differential operator characterized by a distinctive matrix kernel which is used to couple differences of components of a vector field. For an operator possessing an asymmetric kernel comparable to that of the fractional Laplacian, we prove the $L^2$-solvability of the elliptic system in a Bessel potential space using the Fourier transform and \textit{a priori} estimates. This $L^2$-solvability together with the Hille-Yosida theorem is used to prove the well posedness of the wave-type time dependent problem. For the fractional Laplacian kernel we extend the solvability to $L^p$ spaces using classical multiplier theorems. 
\end{abstract}

\section{Introduction}

In this note we report a solvability result for the strongly-coupled system of linear equations
\begin{equation}\label{eq:GeneralKernel:WaveEquation}
\begin{cases}
\p_{tt} \bu(\bx,t) + \bbL \bu(\bx,t) = \bff(\bx,t)\,, &\qquad  (\bx,t) \text{ in } \bbR^d \times [0,T]\,, \\
\bu(\bx,0) = \bu_0(\bx)\,, &\qquad \bx \text{ in } \bbR^d\,, \\
\p_t \bu(\bx,0) = \bv_0(\bx)\,, &\qquad \bx \text{ in } \bbR^d\,,
\end{cases}
\end{equation}
where the vector valued operator $\bbL$ is given by 
\begin{equation}\label{defn-operator}
-\bbL \bu(\bx,t) = \pv \intdm{\bbR^d}{ \shapetensorby \left( \bu(\bx+\by,t)-\bu(\bx,t) - \veps (\bu)(\bx,t) \by \, \chi^{(s)}(\by) \right) \, \rho(\by) }{\by}\,.
\end{equation}
The definition of the operator will be explained later along with precise conditions on $\rho$. 
The system of equations \eqref{eq:GeneralKernel:WaveEquation} is inspired by the equation of motion in linearized bond-based peridynamics \cite{silling2000}, a continuum model in mechanics that uses integral operators in lieu of differential operators to describe physical quantities. 
In the peridynamic model, an elastic material occupying a bounded domain is treated as a complex mass-spring system where any two points $\bx$ and $\by$ are assumed to be interacting through the bond vector $\bx - \by$. When the material is subjected to an external load $\bff$ it undergoes a deformation that maps a point $\bx$ in the domain to the point $\bx + \bu(\bx) \in \bbR^d$, where the vector field $\bu$ represents the displacement field. Under the uniform small strain theory \cite{silling2010} the strain of the bond $\bx - \by$ is given by the nonlocal linearized strain
$
\big( \bu(\bx) - \bu(\by) \big) \cdot \frac{\bx-\by}{|\bx-\by|}.
$
According to the linearized bond-based peridynamic model \cite{silling2010}, the balance of forces is given formally by a strongly coupled system of equations of the form
\begin{equation}\label{eq-PeridynamicSystem-WaveEquation}
\p_{tt} \bu(\bx,t) + \intdm{\bbR^d}{\rho(\bx,\by) \left(\big( \bu(\bx,t) - \bu(\by,t) \big) \cdot \frac{\bx-\by}{|\bx-\by|} \right) \frac{\bx-\by}{|\bx-\by|} }{\by} = \bff(\bx,t)\,.
\end{equation}
The kernel $\rho : \bbR^d \times \bbR^d \to \bbR$ 
encodes the strength and extent of interactions between the material points $\bx$ and $\by$. The kernel may depend on $\bx$, $\by$, their relative position $\bx-\by$ or, in the case of homogeneous isotropic materials, only on the relative distance $|\bx-\by|$. Thus, more general kernels $\rho$ have the potential to model heterogeneous and anisotropic long-range interactions. 

To analyze \eqref{eq-PeridynamicSystem-WaveEquation} we turn to the equilibrium equations for the corresponding steady-state system, given by
\begin{equation}\label{eq-PeridynamicSystem}
\intdm{\bbR^d}{\rho(\bx,\by) \left(\big( \bu(\bx) - \bu(\by) \big) \cdot \frac{\bx-\by}{|\bx-\by|} \right) \frac{\bx-\by}{|\bx-\by|} }{\by} = \bff(\bx)\,.
\end{equation}
The majority of the paper is devoted to obtaining results for \eqref{eq-PeridynamicSystem}, which are then leveraged via semigroups to obtain results for \eqref{eq-PeridynamicSystem-WaveEquation}. This semigroup approach will be explained in detail in Section \ref{sec:TimeDepProb}, and at present we focus on the system \eqref{eq-PeridynamicSystem}.
Mathematical analysis of \eqref{eq-PeridynamicSystem} often considers the case when the kernel is compactly supported, radially symmetric, $\rho(\bx, \by) = \rho(|\bx-\by|)$,  and bounded away from zero in a neighborhood of the origin \cite{Du-Zhou2011, Du-Zhou2010, mengesha2012nonlocalKorn, MengeshaDuElasticity}. The time dependent  linearized equation of motion  \eqref{eq-PeridynamicSystem-WaveEquation}  is also studied in \cite{Emmrich-Weckner2007, Etienne2} as an evolution equation in various spaces when the kernel is radial.   In the above cases, the integral operator is well-defined either as a convolution-type operator, when the kernel is integrable, or in the {\em principal value} sense in when the kernel has strong singularity.

In this work we study models associated with kernels that are translation-invariant but may be rotationally variant. 
To be precise, we assume $\rho(\bx,\by) = \rho(\bx-\by)$, where $\rho$ is not necessarily symmetric. 
This assumption combined with a formal change of variables leads to the nonlocal system
$$
\intdm{\bbR^d}{\rho(\by) \shapetensorby \big( \bu(\bx + \by) - \bu(\bx) \big)}{\by} = -\bff(\bx)\,.
$$ 
The integral operator on the left hand side converges when $\rho\in L^{1}(\mathbb{R}^{d})$. However, when $\rho$ is not integrable, it is not clear that  the integral converges, even when $\bu$ is smooth. We therefore modify the integral operator to ensure that it is well-defined. 
To illustrate the modification, first consider the specific choice of kernel $\rho(\by) := |\by|^{-d-2s}$ for $s \in (0,1)$, and define the integral operator
\begin{equation}\label{DotLap}
-\DotLap^s \bu(\bx) := \pv \intdm{\bbR^d}{\frac{1}{|\by|^{d+2s}} \shapetensorby \big( \bu(\bx+\by)-\bu(\bx) \big)}{\by}\,.
\end{equation}
As indicated, the integral converges in the principal value sense for $\bu$ smooth enough, say $C^2$. It is in fact a fractional analogue of the L\'ame operator in linearized elasticity \cite{MengeshaDuElasticity}.  Since the kernel is radial, we can subtract any odd function in the integrand and get, for example,
$$
-\DotLap^s \bu(\bx) = \pv \intdm{\bbR^d}{\frac{1}{|\by|^{d+2s}} \shapetensorby \Big( \bu(\bx+\by)-\bu(\bx) - \grad \bu(\bx) \by \chi^{(s)}(\by) \Big)}{\by}\,,
$$
where the function $\chi^{(s)}$ is defined as
$$
\chi^{(s)}(\by) = 
\begin{cases}
0 & \text{ if } s \in (0,1/2) \\
\mathds{1}_{B_1}(\by) & \text{ if } s =1/2 \\
1 & \text{ if } s \in (1/2,1)\,. \\
\end{cases}
$$
Because of the presence of the rank-one matrix $\shapetensorby$, the last term can be rewritten, giving
$$
-\DotLap^s \bu(\bx) = \pv \intdm{\bbR^d}{\frac{1}{|\by|^{d+2s}} \shapetensorby \Big( \bu(\bx+\by)-\bu(\bx) - \veps (\bu)(\bx) \by \chi^{(s)}(\by) \Big)}{\by}\,,
$$
where  $\veps (\bu)(\bx)$ is the symmetric part of the gradient matrix given by
\[
\veps (\bu)(\bx) = {1\over 2}(\grad \bu(\bx) + \grad \bu(\bx)^{\intercal}). 
\]
With this motivation at hand, we may now replace the kernel $|\by|^{-d-2s}$ with any translation-invariant kernel $\rho(\by)$ and introduce the integral operator $\bbL$ by
\begin{equation*}
-\bbL \bu(\bx) = \pv \intdm{\bbR^d}{ \shapetensorby \left( \bu(\bx+\by)-\bu(\bx) - \veps (\bu)(\bx) \by \, \chi^{(s)}(\by) \right) \, \rho(\by) }{\by}
\end{equation*}
which is precisely the operator defined in \eqref{defn-operator}. 
With this new modification, it is clear now that even for radially nonsymmetric kernels comparable to $|\by|^{-d-2s}$ the integral converges absolutely for smooth $\bu$.  More generally, the kernel $\rho$ that we will consider in this paper will come from two classes.  

{\bf Class A: Integrable kernels:} $\rho$ is a nonnegative integrable function in $\mathbb{R}^{d}$.
That is,
$\rho(\by) \in L^1(\bbR^d)$. In this case, in the definition of $\bbL$ in \eqref{defn-operator}, we take $\chi^{(s)}(\by) \equiv 0$. Integrable kernels commonly used in applications have compact support.


 {\bf Class B: Nonintegrable kernels}: $\rho$ is a nonnegative singular kernel that is comparable to $|\by|^{-d-2s}$. More precisely, 
we  assume that the kernel $\rho$ is of the form 
$$
\rho(\by) := \frac{m(\by)}{|\by|^{d+2s}} \, \mathds{1}_{\Lambda_r}(\by)\,,
$$
where for some $0<r\leq \infty$ the set $\Lambda_r := \Lambda \cap B_r$ is a truncated double cone $\Lambda := \{ \bx \in \bbR^d \, \big| \, \frac{\bx}{|\bx|} \in \Gamma \cup - \Gamma \} $ corresponding to a given measurable subset $\Gamma \subseteq \bbS^{d-1}$ that has positive Hausdorff measure, and the positive measurable function $m : \bbR^d \to [0,\infty)$ satisfies  
\[
0 < \alpha_1 \leq m(\by) \leq \alpha_2 < \infty\,,
\]
for positive constants $\alpha_1$ and $\alpha_2$. 
We also assume that when $s = 1/2$, $\rho$ satisfies the cancellation condition
\begin{equation}\label{eq-KernelCancellationCondition}
\intdm{\p B_{\mu}}{y_i y_j y_k \rho(\by)}{\sigma(\by)} = 0\,, \quad \forall \, i,j,k \in \{1, 2, \ldots, d \}\,, \quad \forall \mu > 0\,.
\end{equation}
Note that \eqref{eq-KernelCancellationCondition} is always satisfied when $m({\by})$ is an even function. 

To show well posedness of the time dependent problem, we establish the solvability of the linear system 
\begin{equation}\label{eq-GeneralKernel-Equation}
\bbL \bu(\bx) + \lambda \bu(\bx) = \bff(\bx) \quad \text{ in } \bbR^d\,,
\end{equation}
which is the first main goal of this paper.
The constant $\lambda$ is nonnegative and the operator $\bbL$ is assumed to possess a kernel from one of the above classes of kernels. For operators that use kernels in the nonintegrable class, we prove the existence and uniqueness of a strong solution in the Bessel potential space $\bu \in H^{2s,2}$ corresponding to data $\bff \in L^2$ satisfying \eqref{eq-GeneralKernel-Equation} almost everywhere. 
Further, when $\mathbb{L} =\DotLap$,  for each $\bff \in L^p$ there exists a unique $\bu \in H^{2s,p}$ solving \eqref{eq-GeneralKernel-Equation}. 
These and other results will be precisely stated in Section 2.  The proof of the strong $L^2$ solvability result for the equation \eqref{eq-GeneralKernel-Equation} is proved in Section 3. Section 4 contains the proofs leading to the strong $L^p$ solvability for $1 < p < \infty$ of the equation \eqref{eq-GeneralKernel-Equation} when the choice of kernel $\rho(\by) = |\by|^{-d-2s}$ is used. In the last section we show well posedness of a wave equation closely resembling \eqref{eq:GeneralKernel:WaveEquation} as a consequence of the solvability obtained for the steady-state problem. We emphasize that our focus in this paper is on the linear problem. For the well posedness of the nonlinear peridynamic equations of motion we refer to \cite{emmrich2013well,EtiennePuhst2015,Valdinoci-peridynamic}.

\section{Statement of Main Results}
Before we state the main results, let us establish notation and define the relevant function spaces.  
Euclidean balls of radius $r$ centered at $\bx_0$ are denoted $B_r(\bx_0) := \{ \bx \in \bbR^d \, \big| \, |\bx-\bx_0| < r \}$. If the center $\bx_0 = {\bf 0}$, or if the center is clear from context, we omit it and write $B_r$. We denote the Fourier transform $\cF$ by
$$
\cF f(\bfxi) = \widehat{f}(\bfxi) = \intdm{\bbR^d}{e^{-2\pi \iota \bx \cdot \bfxi}f(\bx)}{\bx}\,.
$$
We write the $L^p(\Omega)-$norm of a function as the standard $\Vnorm{\cdot}_{L^p(\Omega)}$, with abbreviation $\Vnorm{\cdot}_{L^p}$ whenever the domain of integration is all of $\bbR^d$.
Let $\cS(\bbR^d)$ and $\cS'(\bbR^d)$ be the space of Schwartz functions and tempered distributions, respectively. Denote the space of $\bbR^d$-valued Schwartz vector fields by $\big[ \cS(\bbR^d) \big]^d$, and its dual by $\big[ \cS'(\bbR^d) \big]^d$.
For $p \in (1,\infty)$ and $s\in (0, 1)$, the Bessel potential space is given by 
$$
\big[ H^{2s,p}(\bbR^d) \big]^d := \left\lbrace \bu \in \big[ \cS'(\bbR^d) \big]^d \, \Big| \, \Big(  \big( 1 + 4 \pi^2 |\bfxi|^2 \big)^s \widehat{\bu} \Big)^{\vee}  \in \big[ L^p(\bbR^d) \big]^d \right\rbrace\,,
$$
with norm
$$
\Vnorm{\bu}_{H^{2s,p}} 
= \Vnorm{\Big( \big(1+4 \pi^s |\bfxi|^2 \big)^s \widehat{\bu} \Big)^{\vee}}_{L^p}\,.
$$
Define the homogeneous space
$$
\big[ \dot{H}^{2s,p}(\bbR^d) \big]^d = \left\lbrace \bu \in  \big[ \cS'(\bbR^d) \big]^d \, \Big| \, \Big( \big( 4 \pi |\bfxi|^2  \big)^s \widehat{\bu} \Big)^{\vee} := \Dss \bu \in \big[  L^p(\bbR^d) \big]^d \right\rbrace\,,
$$
and denote the semi-norm
$
[ \bu ]_{H^{2s,p}(\bbR^d)} = \Vnorm{\Dss \bu }_{L^p(\bbR^d)}.
$
Since $1+ \big( 4 \pi^2 |\bfxi|^2 \big)^s$  can be controlled by $ \big( 1+ 4 \pi^2 |\bfxi|^2 \big)^s$ and vice versa, we have
$$
\Vnorm{\bu}_{H^{2s,p}} \approx \Vnorm{\bu}_{L^p} + [\bu]_{H^{2s,p}}\,.
$$
With these notations and definition at hand, we can now state the first result of the paper. 
%

\begin{theorem}[$L^2$ Solvability for general nonintegrable kernels]\label{thm-GeneralKernel-L2Solvability} Suppose that $\rho$ is in {\bf Class B}. There exists a constant $\lambda_{0} = \lambda_0(d,s,\Lambda,r,\alpha_1) > 0$ such that the following holds: For $\lambda > \lambda_0$ and for any $\bff \in \big[ L^2(\bbR^d) \big]^d$ there exists a unique strong solution $\bu \in \big[ H^{2s,2} (\bbR^d) \big]^d$ to the equation \eqref{eq-GeneralKernel-Equation} satisfying the estimate
\begin{equation}
\Vnorm{\Dss \bu}_{L^2} + \sqrt{\lambda} \Vnorm{\Ds \bu}_{L^2} + \big( \lambda - \lambda_0 \big) \Vnorm{\bu}_{L^2} \leq C \Vnorm{\bff}_{L^2}\,,
\end{equation}
where $C = C(d,s,\Lambda,\alpha_1) > 0$.
\end{theorem}

For the solvability of the system corresponding to kernels in {\bf Class A}, we have the following result. 

\begin{theorem}[$L^2$ Solvability for general integrable kernels]\label{thm:GeneralL1KernelAsymmetric:L2Solvability}
Suppose that $\rho$ is in {\bf Class A}. Then for any $\lambda > 0$ and for any $\bff \in \big[ L^2(\bbR^d) \big]^d$ there exists a unique strong solution $\bu \in \big[ L^2(\bbR^d) \big]^d$ to the equation \eqref{eq-GeneralKernel-Equation} satisfying the estimate
\begin{equation}
\lambda \Vnorm{\bu}_{L^2} \leq C \Vnorm{\bff}_{L^2}\,,
\end{equation}
where $C = C(d, \rho) > 0$.
\end{theorem}

$L^{2}$ solvability of problems of the type \eqref{eq-GeneralKernel-Equation} have been considered in \cite{Du-Zhou2011}. There, the $H^{2s,2}$ regularity of appropriately-defined weak solutions of \eqref{eq-GeneralKernel-Equation} with $\lambda = 1$ and $\rho$ positive and radially symmetric was obtained via the Fourier transform and inverting the positive definite operator $\bbL + \bbI$, where $\bbI$ is the $d \times d$ identity matrix. Without much difficulty, the same technique could be used to prove existence and uniqueness of strong solutions of \eqref{eq-GeneralKernel-Equation} with $\rho$ positive and radially symmetric. In this paper we prove this result of well posedness -- but for a wider class of kernels -- using the Fourier transform in a slightly different way, via \textit{a priori} estimates and the celebrated method of continuity \cite{krylov2008lectures}. The novelty here is that no symmetry assumptions of any kind are made on the kernel. We also assume that the kernel can vanish on a substantial set, specifically outside of a double cone whose apex is at the origin. The work in this paper may begin the mathematical analysis for a system of equations that potentially model anisotropic interactions in materials, generalizing the model equations in \cite{Du-Zhou2011}.

The $L^p$ theory, on the other hand, for equilibrium systems of equations of the type \eqref{eq-GeneralKernel-Equation} is relatively unstudied. In a recent work, the Dirichlet problem for equations resembling \eqref{eq-PeridynamicSystem} has been studied in \cite{KassmannMengeshaScott} where the authors  used Hilbert space techniques to prove existence and uniqueness of weak solutions satisfying a complementary condition (a ``volume constraint problem" in peridynamics). In the current work, in place of weak solutions, we consider strong solutions solving an equation almost everywhere. The equation is also posed on all of $\bbR^d$ rather than on a bounded domain. The Fourier matrix symbol associated to the operator $\bbL$ is
\begin{equation}\label{eq-FourierSymbolOfL}
\bbM(\bfxi) := \intdm{\bbR^d}{\shapetensorby \left( \e^{2 \pi \iota \by \cdot \bfxi} - 1 - 2 \pi \iota \bfxi \cdot \by \, \chi^{(s)}(\by) \right) \rho(\by) }{\by}\,,
\end{equation}
which in general lacks the differentiability necessary to apply classical multiplier theorems. The $L^p$ results obtained in this work are for the specific kernel $|\by|^{-d-2s}$. This choice of kernel allows for the use of Fourier multiplier theorems, specifically the Marcinkiewicz multiplier theorem. The following theorem states the second result of the paper. 

\begin{theorem}[$L^p$ Solvability for Specific Kernels]\label{thm-FracLapKernel-LpSolvability}
For $1 < p < \infty$ and corresponding to any $\bff \in L^p$ and for any $\lambda > 0$ the equation
\begin{equation}\label{eq-FracLapKernel-Equation}
\DotLap^s \bu + \lambda \bu = \bff \quad \text{ in } \bbR^d
\end{equation}
has a unique strong solution $\bu \in \big[ H^{2s,p}(\bbR^d) \big]^d$  satisfying the estimate
\begin{equation}\label{eq-FracLapKernel-LpAPrioriEstimate}
\Vnorm{\bu}_{H^{2s,p}} \leq C \Vnorm{\bff}_{L^p}\,,
\end{equation}
where $C = C(d,s,p,\lambda) > 0$.
\end{theorem}

We finally remark on the case of scalar operators related to $\bbL$, for example
$$
\cL u(\bx) := \pv \intdm{\bbR^d}{\Big( u(\bx + \by) - u(\bx) - \grad u(\bx) \, \by \chi^{(s)}(\by) \Big) \rho(\by)}{\by}\,.
$$
The operator $\cL$ is a nonlocal elliptic operator associated to a stochastic process called a Markov jump process, or specifically a L\'evy process. 
For instance, see \cite{banuelos2007levy} 
for some earlier work. The solvability in Sobolev and H\"older spaces for parabolic equations associated to the elliptic problem $\cL u + \lambda  u$ has been considered in \cite{Mikulevicius2014} via a probabilistic approach. L\'evy processes were studied using Fourier multipliers in \cite{banuelos2007levy}. The Fourier multipliers studied in \cite{banuelos2007levy} are quotients of symbols consisting of integrals. Because of their particular form, \textit{a priori} estimates for solutions to scalar-valued equations are easily obtained as a corollary. However, the results of \cite{banuelos2007levy} do not directly apply to the matrix symbols necessary for the analogous \textit{a priori} estimates of solutions to \eqref{eq-GeneralKernel-Equation}, as matrix inverses take the place of quotients in the symbols.

Our work is inspired by the paper \cite{DoKi12} where it was shown that the equation $\cL u + \lambda u = f$ -- with $\rho$ merely measurable and satisfying a certain cancellation condition -- is $L^p$-solvable. 
The paper applies maximum principle techniques to obtain important estimates. For the system \eqref{eq-GeneralKernel-Equation} an analogous $L^p$ solvability result remains unclear. The maximum principle techniques in \cite{DoKi12} do not apply to systems of the type \eqref{eq-GeneralKernel-Equation}.

With Theorem \ref{thm-GeneralKernel-L2Solvability} and  Theorem \ref{thm:GeneralL1KernelAsymmetric:L2Solvability}, we show the well posedness of the nonlocal hyperbolic system of wave equations
\begin{equation}\label{eq:WaveEquation-MR}
\begin{cases}
\p_{tt} \bu + \bbL \bu  + \lambda \bu = \bff\,, &\qquad  (\bx,t) \text{ in } \bbR^d \times [0,T]\,, \\
\bu(\bx,0) = \bu_0(\bx)\,, &\qquad \bx \text{ in } \bbR^d\,, \\
\p_t \bu(\bx,0) = \bv_0(\bx)\,, &\qquad \bx \text{ in } \bbR^d\,,
\end{cases}
\end{equation}
The main existence and uniqueness result we prove is the following.

\begin{theorem}\label{thm:WaveEqnSolvability} 
Let $T>0$, and let $\bff \in C^1 \left( [0,T] ; \big[ L^2(\bbR^d) \big]^d \right)$, $\bu_0 \in \big[ H^{2s,2}(\bbR^d) \big]^d$, and $\bv_0 \in \big[ H^{s,2}(\bbR^d) \big]^d$. 
Suppose that $\rho$ is an even function that is in {\bf Class B}. Suppose that $\lambda$ is a nonnegative constant satisfying $\lambda > \lambda_0 - 1$, where $\lambda_0$ is the constant appearing in Theorem \ref{thm-GeneralKernel-L2Solvability}. Then there exists a unique solution
\begin{equation*}
\bu \in 
	C \left( [0,T] ; \big[ H^{2s,2}(\bbR^d) \big]^d \right)
	\cap
	C^1 \left( [0,T] ; \big[ H^{s,2}(\bbR^d) \big]^d \right)
	\cap
	C^2 \left( [0,T] ; \big[ L^2(\bbR^d) \big]^d \right)
\end{equation*}
to \eqref{eq:WaveEquation-MR}. In the case that $\bff(\bx,t) =  \bff_0(\bx)$, then the solution $\bu$ additionally satisfies the conservation law 
\begin{equation}\label{eq:ConservationLaw-MR}
\begin{split}
\Vnorm{\p_t \bu(t)}_{L^2(\bbR^d)}^2 &+ \iintdm{\bbR^d}{\bbR^d}{\rho(\bx-\by) \left( \big( \bu(\bx,t) - \bu(\by,t) \big) \cdot \frac{\bx-\by}{|\bx-\by|} \right)^2 }{\by}{\bx}  \\
&+ \lambda \Vnorm{\bu(t)}_{L^2(\bbR^d)}^2 + 2 \intdm{\bbR^d}{\Vint{\bff_0,\bu(t)}}{\bx} \\
= \Vnorm{\p_t \bu_0}_{L^2(\bbR^d)}^2 &+ \iintdm{\bbR^d}{\bbR^d}{\rho(\bx-\by) \left( \big( \bu_0(\bx) - \bu_0(\by) \big) \cdot \frac{\bx-\by}{|\bx-\by|} \right)^2 }{\by}{\bx}  \\
&+ \lambda \Vnorm{\bu_0}_{L^2(\bbR^d)}^2 + 2 \intdm{\bbR^d}{\Vint{\bff_0,\bu_0}}{\bx}\,.
\end{split}
\end{equation}
If $\lambda_0 < 1$, then $\lambda$ can be taken to be $0$, that is, solving the system of equation given in \eqref{eq:GeneralKernel:WaveEquation}.
\end{theorem}
An analogous result when $\rho$ is radial and integrable can be found in \cite{Du-Zhou2011}. 

\section{$L^{2}$ solvability}
In this section we will use the method of continuity to prove the solvability of the coupled system of nonlocal equations.   The general result of the method of continuity is stated as follows. 
\begin{proposition}[The method of continuity]\label{M-o-C}
Suppose that $X$ is a Banach space and $V$ is a normed space.  Suppose also $T_0, T_1 : X\to V$ are bounded linear operators. Assume that there exists $C>0$ such that for $T_t = (1-t)T_0 + t T_1$ we have 
\[
\|x\|_{X} \leq C\|T_t x\|_{V},\quad \forall t\in[0, 1],  \forall x \in X. 
\]
Then $T_0$ is onto if and only if $T_1$ is onto. 
\end{proposition}
In our case, we take the operator $T_1$ to be $\mathbb{L} + \lambda \mathbb{I}$, where $\mathbb{L}$ is as defined in \eqref{defn-operator}.  
In the case of nonintegrable kernels, $T_0$ will be $\alpha_1 \DotLap^s + \lambda$ for some appropriately chosen $\lambda>0$, since the solvability of the system $\alpha_1 \DotLap^s \bu + \lambda \bu = \bff$ in $\big[ H^{2s,2}(\bbR^d) \big]^d$ is well-known,  see \cite{Du-Zhou2011} or Theorem \ref{thm-FracLapKernel-LpSolvability} for $p=2$. For the case of integrable kernels, we take the operator $T_0$ to be the same operator  $\mathbb{L} + \lambda \mathbb{I}$ but one with radial kernel where invertibility of the operator will be proved. In both cases, in order to apply the method of continuity we need to show boundedness of the operators and obtain {\em a priori} estimates for the corresponding operator $T_t$. 

\subsection{The nonintegrable case.}
Let us introduce the parametrized linear operators 
\begin{equation*}
- \bbL_t \bu(\bx) := \pv \intdm{\bbR^d}{ \shapetensorby \left( \bu(\bx+\by)-\bu(\bx) - \veps (\bu)(\bx) \by \, \chi^{(s)}(\by) \right) \, \rho_t(\by) }{\by}\,,
\end{equation*}
and the kernel $\rho_t$ is defined as
\begin{equation*}
\rho_t(\by) := t \, \rho(\by) + (1-t)\frac{\alpha_1}{|\by|^{d+2s}}\,.
\end{equation*}
Notice that $- \bbL_0 = \alpha_1\DotLap^s$ and $\bbL_1 = \bbL$.  In order to use the method of continuity, we need to show that the operators $\bbL_0$, and $\bbL_1$ are bounded from $\big[ H^{2s,2}(\bbR^d) \big]^d$ to $ \big[ L^2(\bbR^d) \big]^d$ and establish  the \textit{a priori} estimate 
for solutions of $\bbL_t \bu + \lambda \bu = \bff$ for $t \in [0,1]$. 
First let us prove that the operators are continuous. 
\begin{lemma}[Continuity of the parametrized operator: nonintegrable case]\label{thm-GeneralKernel-L2Continuity-Ltu} Suppose that $\rho$ is in {\bf Class B}. 
Then for each $t\in [0, 1]$ the operator $\bbL_t : \big[ H^{2s,2}(\bbR^d) \big]^d \to \big[ L^2(\bbR^d) \big]^d$ is continuous. More precisely, there exists a constant $C>0$ such that for all $t\in [0, 1]$,  and for all $\bu \in \big[ \cS(\bbR^d) \big]^d$ we have the estimate
\begin{equation}
\Vnorm{\bbL_t \bu}_{L^2} \leq C \Vnorm{\bu}_{H^{2s,2}}\,.
\end{equation}
\end{lemma}

\begin{proof}
%
Recalling that the symbol of $\mathbb{L}$ is the matrix symbol $\bbM(\bfxi)$ given by \eqref{eq-FourierSymbolOfL}, to prove continuity it suffices to show that there exists a constant $C = C(d,s,\alpha_2)$ such that
\begin{equation}
|\bbM_t(\bfxi)| \leq C (2 \pi |\bfxi|)^{2s}\,,
\end{equation}
where $\bbM_t(\bfxi)$ is the matrix symbol associated to $\bbL_t$. To that end, if  $s\in (0, {1\over 2})$, by definition $\mathbb{M}_t(\bfxi) = \intdm{\bbR^d}{\shapetensorby \big(e^{2\pi\imath\by\cdot \bfxi} - 1\big)\rho_t(\by) }{\by}$.  By using the upper bound $\rho_t(\by) \leq \alpha_2 |\by|^{-d-2s}$, we obtain 
\begin{equation*}
|\bbM_t(\bfxi)| \leq  \alpha_2 \intdm{\bbR^d}{\frac{1}{|\by|^{d+2s}}  \Big|e^{2\pi\imath\by\cdot \bfxi} - 1\Big|}{\by}. 
\end{equation*}
 Now a simple calculation shows that $  |e^{2\pi\imath\by\cdot \bfxi} - 1| =  \sqrt{2(1-\cos ({2\pi\by\cdot \bfxi}))}$.  Using this and by making the substitution $\bh = 2 \pi |\bfxi| \by$, we obtain 
\[
\begin{split}
|\bbM_t(\bfxi)| &\leq  \alpha_2\intdm{\bbR^d}{\frac{1}{|\by|^{d+2s}}\sqrt{2(1-\cos ({2\pi\by\cdot \bfxi}))}} {\by} \\
&\leq \alpha_2 (2 \pi |\bfxi|)^{2s} \intdm{\bbR^d}{\frac{\sqrt{2}}{|\bh|^{d+2s}} \sqrt{1-\cos \left( \bh \cdot \frac{\bfxi}{|\bfxi|}\right)}} {\bh}\,.
\end{split}
\]
Notice that the last integral is uniformly bounded in ${\bfxi}$, since $1-\cos(a) \approx |a|^2$ for $a$ near $0$ and $s < 1/2$.

When $s = 1/2$, we have by the cancellation condition \eqref{eq-KernelCancellationCondition} on $\rho$ that
\begin{equation*}
\bbM_t(\bfxi) = \intdm{\bbR^d}{\shapetensorby \bigg( \e^{2 \pi \iota \by \cdot \bfxi} - 1 - 2 \pi \iota \bfxi \cdot \by \, \mathds{1}_{B_{1/ (2 \pi |\bfxi|)}}(\by) \bigg) \, \rho_t(\by) }{\by}\,.
\end{equation*}
Then by the same substitution $\bh = 2 \pi |\bfxi| \by$ we have
\begin{equation*}
\begin{split}
|\bbM_t(\bfxi)| &\leq \alpha_2 \intdm{\bbR^d}{ \frac{1}{|\by|^{d+1}} \Big( | 1 - \cos(2 \pi \by \cdot \bfxi)| + |\sin(2 \pi \by \cdot \bfxi ) - 2 \pi \bfxi \cdot \by \, \mathds{1}_{B_{1/ (2 \pi |\bfxi|)}}(\by)|  \Big) }{\by} \\
	&\leq \alpha_2 (2 \pi |\bfxi|)^{2s} \intdm{\bbR^d}{ \frac{1}{|\bh|^{d+1}} \left( \left| 1 - \cos \left(\bh \cdot \frac{\bfxi}{|\bfxi|} \right) \right| + \left| \sin \left(\bh \cdot \frac{\bfxi}{|\bfxi|} \right) -  \frac{\bfxi}{|\bfxi|} \cdot \bh \, \mathds{1}_{B_1}(\bh) \right|  \right)}{\bh} \\
	&\leq C (2 \pi |\bfxi| )^{2s}\,.
\end{split}
\end{equation*}
The integral is again uniformly bounded in $\bfxi$ since $|\sin(a) - a | \approx |a|^3$ for $a$ near $0$ and since the term $\frac{\bfxi}{|\bfxi|} \cdot \bh \mathds{1}_{B_1}(\bh)$ vanishes for $|\bh| \geq 1$.

When $s \in ({1\over 2}, 1)$, again by using the upper bound on $\rho$ and then making the substitution $\bh = 2 \pi |\bfxi| \by$,
\begin{align*}
|\bbM_t(\bfxi)| &\leq  \alpha_2 \intdm{\bbR^d}{\frac{1}{|\by|^{d+2s}} \bigg( \big| 1 - \cos(2 \pi \by \cdot \bfxi ) \big| + \left| \sin(2 \pi \by \cdot \bfxi) - 2 \pi \by \cdot \bfxi  \right|  \bigg) }{\by} \\
	&=  \alpha_2 (2 \pi |\bfxi|)^{2s} \intdm{\bbR^d}{\frac{1}{|\bh|^{d+2s}} \bigg( \left| 1 - \cos \left( \bh \cdot \frac{\bfxi}{|\bfxi|} \right) \right| + \left| \sin \left( \bh \cdot \frac{\bfxi}{|\bfxi|} \right) - \bh \cdot \frac{\bfxi}{|\bfxi|}  \right|  \bigg) }{\bh} \\
&\leq C (2 \pi |\bfxi|)^{2s}\,,
\end{align*}
since the integral is once again uniformly bounded in $\bfxi$. 
\end{proof}

The next theorem gives us \textit{a priori} estimates for solutions of the system that imply uniqueness of strong solutions. It is also the key estimate to implement  the  method of continuity to prove existence of a solution. 

\begin{lemma}[\textit{A priori} estimates for parametrized operator: nonintegrable case]\label{lma:GeneralKernel:L2AprioriEstimate:Ltu} Suppose that $\rho$ is in {\bf Class B}.  
Let $\lambda > 0$ be a constant and let $\bff \in \big[ L^2(\bbR^d) \big]^d$. Suppose $\bu \in \big[ H^{2s,2}(\bbR^d) \big]^d$ satisfies the equation $\bbL_t \bu + \lambda \bu = \bff$ in $\bbR^d$. Then there exist constants $N_1 = N_1(d,s,\Lambda, \alpha_1) > 0$ and $N_2 = N_2(\alpha_1,\Lambda,r)>0$ such that
\begin{equation}\label{eq-GeneralKernel-L2AprioriEstimate-Ltu}
\Vnorm{\Dss \bu}_{L^2} + \sqrt{\lambda} \Vnorm{\Ds \bu}_{L^2} + \lambda \Vnorm{\bu}_{L^2} \leq N_1 \Big( \Vnorm{\bff}_{L^2} + N_2 \Vnorm{\bu}_{L^2} \Big)\,.
\end{equation}
As a consequence, 
\begin{equation}\label{eq-GeneralKernel-FullyEquippedAPrioriEstimate-Ltu}
\Vnorm{\Dss \bu}_{L^2} + \sqrt{\lambda} \Vnorm{\Ds \bu}_{L^2} + \big(  \lambda - N_1 N_2 \big) \Vnorm{\bu}_{L^2} \leq N_1 \Vnorm{\bff}_{L^2}\,.
\end{equation}
\end{lemma}
\begin{proof}
Using the continuity of $\mathbb{L}_t$, Lemma \ref{thm-GeneralKernel-L2Continuity-Ltu} and the fact that  $\big[ C^{\infty}_c(\bbR^d) \big]^d$ is dense in $\big[ H^{2s,2}(\bbR^d) \big]^d$, it suffices to show \eqref{eq-GeneralKernel-L2AprioriEstimate-Ltu} for $\bu \in \big[ C^{\infty}_c(\bbR^d) \big]^d$. 
To accomplish this, we introduce the modified kernel
\begin{equation*}
\widetilde{\rho}_t(\by) := \rho_{t}(\by) + \frac{t\alpha_1 }{|\by|^{d+2s}}\chi_{ \Lambda \cap \complement B_r}(\by)=
	\begin{cases}
		\frac{t\alpha_1 }{|\by|^{d+2s}} +\rho_{t}(\by)&\quad \by \in \Lambda \cap \complement B_r\,, \\
		\rho_t(\by) &\quad \text{otherwise}\,, \\
	\end{cases}
\end{equation*}
and the associated operator $\widetilde{\bbL}_t$; i. e.  $\bbL_t$ with $\widetilde{\rho}_t$ in place of $\rho_t$. Note that\begin{equation}
\widetilde{\rho}_t(\by) \geq 0 \text{ on } \bbR^d\,, \qquad  \frac{\alpha_1}{|\by|^{d+2s}} \leq \widetilde{\rho}_t(\by) \leq \frac{\alpha_2}{|\by|^{d+2s}} \quad \text{for} \,\, \by \in \Lambda\,, t\in [0, 1]\,.
\end{equation}
Then we have that 
\begin{align*}
\widetilde{\bbL}_t \bu(\bx) &= \bbL_t \bu(\bx) + \intdm{\bbR^d}{\shapetensorby \Big( \bu(\bx+\by) - \bu(\bx) - \veps(\bu)(\bx) \by \, \chi^{(s)}(\by) \Big) \frac{t\alpha_1  \, \mathds{1}_{\Lambda \cap \complement B_r}(\by)}{|\by|^{d+2s}}}{\by} \\
&= \bbL_t \bu(\bx) + \intdm{\bbR^d}{\shapetensorby \Big( \bu(\bx+\by) - \bu(\bx) \Big) \frac{t\alpha_1 \, \mathds{1}_{\Lambda \cap \complement B_r}(\by)}{|\by|^{d+2s}}}{\by} \\
&:= \bbL_t \bu(\bx) + \bbL_t^c \bu(\bx)\,.
\end{align*}
The second equality follows from the definition of the function $\chi^{s}(\by)$, the cancellation condition \eqref{eq-KernelCancellationCondition}, and the fact that the function $\mathds{1}_{\Lambda \cap \complement B_r}(\by)$ is an even function. 
 It follows by the triangle inequality that 
\begin{equation}\label{eq-GeneralKernel-L2AprioriEstimateProof-I-Ltu}
\Vnorm{\widetilde{\bbL}_t \bu + \lambda \bu}_{L^2} \leq \Vnorm{\bbL_t \bu + \lambda \bu }_{L^2} +  \Vnorm{\bbL_t^c \bu}_{L^2}\,.
\end{equation}
The operator $\bbL_t^c \bu(\bx)$ has a Fourier matrix symbol that belongs to $\big[ L^{\infty}(\bbR^d) \big]^{d \times d}$, and is therefore bounded from $\big[L^2(\bbR^d) \big]^d$ to $\big[L^2(\bbR^d) \big]^d$.
To be precise, 
\begin{align*}
\left| \widehat{\bbL_t^c \bu}(\bfxi) \right| &= \left| t\alpha_1 \intdm{\Lambda \cap \complement B_r}{\shapetensorby \frac{\e^{2 \pi \iota \by \cdot \bfxi}-1}{|\by|^{d+2s}}}{\by} \, \,  \widehat{\bu}(\bfxi) \right| \\
&\leq 2 \alpha_1\, \int_{(\Gamma \cup - \Gamma)} d\sigma(\nu) \intdmt{r}{\infty}{\tau^{-1-2s}}{\tau} \, \,  |\widehat{\bu}(\bfxi)| \\
&:= N_2 |\widehat{\bu}(\bfxi)|\,,
\end{align*}
where $d\sigma$ is the surface measure on $\bbS^{d-1}$. Thus, $\Vnorm{\bbL_t^c \bu}_{L^2} 
 \leq N_3  \|{\bf u}\|_{L^{2}}$. 
Next we need to show that
\begin{equation}\label{eq-GeneralKernel-L2AprioriEstimateProof-II-Ltu}
N_1 \Vnorm{\widetilde{\bbL}_t \bu + \lambda \bu}_{L^2} \geq  \Vnorm{\Dss \bu}_{L^2} + \sqrt{\lambda} \Vnorm{\Ds \bu}_{L^2} + \lambda \Vnorm{\bu}_{L^2}\,.
\end{equation}
Note that
\[
\Vnorm{\widetilde{\bbL}_t \bu + \lambda \bu}^{2}_{L^2}  = \intdm{\bbR^d}{|\widetilde{\bbL}_t \bu|^2}{\bx} + \lambda\intdm{\bbR^d}{\Vint{\widetilde{\bbL}_t \bu, \bu}}{\bx} + \lambda^{2}\|{\bf u}\|^{2}_{L^{2}}\,,
\]
and thus to show \eqref{eq-GeneralKernel-L2AprioriEstimateProof-II-Ltu} it suffices to estimate the first two terms.  To that end, we begin writing $\widetilde{\rho}_t$ as a sum of its even and odd parts:
$$
\widetilde{\rho}_t = \widetilde{\rho}_t^e + \widetilde{\rho}_t^o, \quad \widetilde{\rho}_t^e(\by) = \frac{1}{2}(\widetilde{\rho}_t(\by)+\widetilde{\rho}_t(-\by)), \quad \widetilde{\rho}_t^o(\by) = \frac{1}{2}(\widetilde{\rho}_t(\by) - \widetilde{\rho}_t(-\by))\,.
$$
$\widetilde{\rho}_t^e$ satisfies the same assumptions as $\widetilde{\rho}_t$. Let $\widetilde{\bbL}_t^e$ and $\widetilde{\bbL}_t^o$ be the operators with kernels $\widetilde{\rho}_t^e$ and $\widetilde{\rho}_t^o$ respectively. Observe that $\widetilde{\mathbb{L}}_t = \widetilde{\bbL}_t^e  +\widetilde{\bbL}_t^o $ and that for any ${\bf u}\in \big[ C_c^{\infty}(\mathbb{R}^{d}) \big]^d$ we have 
$$
\intdm{\bbR^d}{\Vint{\widetilde{\bbL}_t^e \bu , \widetilde{\bbL}_t^o \bu}}{\bx} = 0\,.
$$
Thus,
\begin{align*}
\intdm{\bbR^d}{|\widetilde{\bbL}_t \bu|^2}{\bx} &\geq \intdm{\bbR^d}{|\widetilde{\bbL}_t^e \bu(\bx)|^2}{\bx} = \intdm{\bbR^d}{|\widehat{\widetilde{\bbL}_t^e \bu}(\bfxi)|^2}{\bfxi} 
= \intdm{\bbR^d}{|\widetilde{\bbM}_t^e(\bfxi)\widehat{\bu}(\bfxi)|^2}{\bfxi}\,,
\end{align*}
where
$$
\widetilde{\bbM}_t^e(\bfxi) := \intdm{\bbR^d}{ \shapetensorby \big( 1- \cos(2 \pi \by \cdot \bfxi) \big) \widetilde{\rho}_t^e(\by)  }{\by}\,.
$$
For every $\bfxi \in \bbR^d$, the matrix $\widetilde{\bbM}_t^e(\bfxi)$ is symmetric with real entries, and therefore, the least of the eigenvalues of $\widetilde{\bbM}_t^e(\bfxi)$ is given by  $\min_{\bv \in \bbS^{d-1}} \bv^{\intercal} \widetilde{\bbM}_t^e(\bfxi) \bv$. We will estimate the smallest eigenvalue from below as as function of $\bfxi$.  
By making the substitution $\bh = 2 \pi |\bfxi| \by$, we see that 
\begin{align*}
\min_{\bv \in \bbS^{d-1}} \bv^{\intercal} \widetilde{\bbM}_t^e(\bfxi) \bv  &= \min_{\bv \in \bbS^{d-1}} \intdm{\bbR^d}{ \big( 1-\cos(2 \pi \by \cdot \bfxi) \big) \left| \bv \cdot \frac{\by}{|\by|} \right|^2 \widetilde{\rho}_t^e(\by)}{\by} \\
&\geq \min_{\bv \in \bbS^{d-1}} \alpha_1 \intdm{\Lambda}{\frac{\big( 1-\cos(2 \pi \by \cdot \bfxi) \big)}{|\by|^{d+2s}} \left| \bv \cdot \frac{\by}{|\by|} \right|^2 }{\by} \\
&= \min_{\bv \in \bbS^{d-1}} \alpha_1 \big( 2 \pi |\bfxi| \big)^{2s} \intdm{\Lambda}{ \frac{1 - \cos \left( 2 \pi \bh \cdot \frac{\bfxi}{|\bfxi|} \right)}{|\bh|^{d+2s}}  \left| \bv \cdot \frac{\bh}{|\bh|} \right|^2 }{\bh} \\
&= \min_{\bv \in \bbS^{d-1}} \alpha_1 \big( 2 \pi |\bfxi| \big)^{2s} \Psi \left( \frac{\bfxi}{|\bfxi|} \,, \bv \right)\,,
\end{align*}
where $\Psi : \bbS^{d-1} \times \bbS^{d-1} \to \bbR$ defined by $\Psi(\boldsymbol{\eta}, \bv) := \intdm{\Lambda}{ \frac{1 - \cos \left( 2 \pi \bh \cdot \boldsymbol{\eta} \right)}{|\bh|^{d+2s}}  \left| \bv \cdot \frac{\bh}{|\bh|} \right|^2 }{\bh} $. The lower bound $\widetilde{\rho}_t^e(\by) \geq  \alpha_1$ can be used since the integrand is nonnegative. Since $\Lambda$ is a double cone with apex at the origin, $\Psi$ is a positive function. $\Psi$ is also clearly continuous on the compact set $\bbS^{d-1} \times \bbS^{d-1}$ so there exists a positive constant $C = C(d,s,\Lambda)$ such that
$
\min\limits_{\boldsymbol{\eta}, \bv \in \bbS^{d-1}} \Psi(\boldsymbol{\eta},\bv) = C > 0\,.
$  We use this to estimate the least eigenvalue of the matrix.
We conclude that for any $\bfxi \neq 0$, we have $\min\limits_{\bv \in \bbS^{d-1}} \bv^{\intercal} \widetilde{\bbM}_t^e(\bfxi) \bv \geq C \big( 2 \pi |\bfxi| \big)^{2s}$. 
We also use this lower bound to estimate the smallest eigenvalue of $(\widetilde{\bbM}_t^e)^{\intercal}(\bfxi) \, \widetilde{\bbM}_t^e(\bfxi)$ from below. Since $\widetilde{\bbM}_t^e(\bfxi)$ is symmetric, 
\begin{align*}
\min_{\bv \in \bbS^{d-1}} \bv^{\intercal} (\widetilde{\bbM}_t^e)^{\intercal}(\bfxi) \, \widetilde{\bbM}_t^e(\bfxi) \bv &= \min \{ \beta^2 : \beta \text{ is an eigenvalue of } \widetilde{\bbM}_t^e(\bfxi) \}\\
&\geq C(d,s,\Lambda,\alpha_1) \big| \big( 2 \pi |\bfxi| \big)^{2s} \big|^2\,.
\end{align*}
Therefore, 
\begin{equation}\label{eq-L2AprioriEstimateInProof1-Ltu}
\begin{split}
\intdm{\bbR^d}{|\widetilde{\bbL}_t \bu |^2}{\bx} &= \intdm{\bbR^d}{\Vint{\widetilde{\bbM}_t^e(\bfxi)\widehat{\bu}(\bfxi),\widetilde{\bbM}_t^e(\bfxi)\widehat{\bu}(\bfxi)}}{\bfxi} \\
&\geq C \intdm{\bbR^d}{\left| \big( 2 \pi |\bfxi| \big)^{2s} \widehat{\bu}(\bfxi) \right|^2}{\bfxi} = C \Vnorm{\Dss \bu}_{L^2}^2\,.
\end{split}
\end{equation}
Similarly, by applying Plancherel's theorem and noting that the matrix $\widetilde{\bbM}_t(\bfxi) := \int_{\bbR^d} \left( \frac{\by \otimes \by}{|\by|^2} \right) \big( \cos(2 \pi \bfxi \cdot \by) -1 \big) \widetilde{\rho}_t (\by) \, \rmd \by$ is symmetric,
\begin{align*}
\intdm{\bbR^d}{\Vint{\widetilde{\bbL}_t \bu, \bu}}{\bx}
	&= \intdm{\bbR^d}{\Vint{\widehat{\widetilde{\bbL}_t \bu}, \overline{\widehat{\bu}}}}{\bfxi} \\
	&= \intdm{\bbR^d}{ \Re \Vint{\widehat{\widetilde{\bbL}_t \bu}, \overline{\widehat{\bu}}}}{\bfxi} \\
	&= \intdm{\bbR^d}{\Vint{\widetilde{\bbM}_t(\bfxi) \widehat{\bu}(\bfxi), \widehat{\bu}(\bfxi) } }{\bfxi} \\
&\geq \intdm{\bbR^d}{|\widehat{\bu}(\bfxi)|^2 \, \min_{\bv \in \bbS^{d-1}} \bv^{\intercal} \widetilde{\bbM}_t(\bfxi) \bv}{\bfxi} \\
&\geq C \intdm{\bbR^d}{\big( 2 \pi |\bfxi| \big)^{2s} |\widehat{\bu}(\bfxi)|^2}{\bfxi}\,.
\end{align*}
It then follows that 
\begin{equation}\label{eq-L2AprioriEstimateInProof2-Ltu}
\intdm{\bbR^d}{\Vint{\widetilde{\bbL}_t \bu, \bu}}{\bx} \geq C \intdm{\bbR^d}{|\Ds \bu|^2}{\bx}\,.
\end{equation}
The result \eqref{eq-GeneralKernel-L2AprioriEstimateProof-II-Ltu} follows from the equality
\begin{align*}
\intdm{\bbR^d}{\left| \widetilde{\bbL}_t \bu + \lambda \bu \right|^2}{\bx} = \intdm{\bbR^d}{|\bff|^2}{\bx}
\end{align*}
and the estimates \eqref{eq-L2AprioriEstimateInProof1-Ltu} and \eqref{eq-L2AprioriEstimateInProof2-Ltu} above. Then \eqref{eq-GeneralKernel-L2AprioriEstimate-Ltu} follows from \eqref{eq-GeneralKernel-L2AprioriEstimateProof-I-Ltu} and \eqref{eq-GeneralKernel-L2AprioriEstimateProof-II-Ltu}.
\end{proof}

\begin{remark}
Constants $N_1$ and $N_2$ satisfying \eqref{eq-GeneralKernel-L2AprioriEstimate-Ltu} can be found explicitly. For example, the constant $N_2$ can be chosen to be
\begin{equation}
N_2 = \frac{1+ \alpha_1}{s} \, \sigma(\Gamma \cup - \Gamma) \, r^{-2s} \,.
\end{equation}
It is clear that $N_2 \to 0$ as $r \to \infty$.
\end{remark}


The proof of Theorem \ref{thm-GeneralKernel-L2Solvability} now follows. 
\begin{proof}[Proof of Theorem \ref{thm-GeneralKernel-L2Solvability}]
We will apply Proposition \ref{M-o-C}.  
Pick $\lambda_0 = N_1N_2$. Then for any $\lambda >\lambda_0$, take $T_0 = \alpha_1 \DotLap  + \lambda \mathbb{I}$, $T_{1} = \mathbb{L} + \lambda\mathbb{I}$, and therefore for $t\in [0, 1]$, $T_{t} = (1-t)T_0  + t T_{1} = \mathbb{L}_t + \lambda \mathbb{I}$, where $\mathbb{L}_t$ is as defined at the beginning of this subsection. Now the continuity and {\em a priori} estimates assumptions of Proposition \ref{M-o-C} are satisfied by Lemma \ref{thm-GeneralKernel-L2Continuity-Ltu} and Lemma \ref{lma:GeneralKernel:L2AprioriEstimate:Ltu}. As a consequence $ \mathbb{L} + \lambda\mathbb{I}$ is onto if and only if $\alpha_1 \DotLap  + \lambda \mathbb{I}$ is onto. The latter is proved to be the case in Theorem \ref{thm-FracLapKernel-LpSolvability} for $p=2$ or see \cite{Du-Zhou2011} .  
\end{proof}
%
%

\subsection{The integrable case.}

For integrable kernels, we begin with a special case by assuming that $\rho$ is radially symmetric.

\begin{theorem}\label{thm:Solvability:IntegrableKernel:Symmetric}
Suppose $\rho$ is a nonnegative, integrable and  radially symmetric function that is not identically 0. For every $\lambda > 0$ and for any $\bff \in \big[ L^2(\bbR^d) \big]^d$ there exists a unique strong solution $\bu \in \big[ L^2(\bbR^d) \big]^d$ to $\mathbb{L}{\bf u} + \lambda {\bf u} = {\bf f}$  satisfying the estimate
\begin{equation}\label{eq:L1SymmetricKernel:EnergyEstimate}
\Vnorm{\bu}_{L^2} \leq  \lambda^{-1} \Vnorm{\bff}_{L^2}\,.
\end{equation}
\end{theorem}

\begin{proof}
To show existence, we prove that the Fourier matrix multiplier $\bbM(\bfxi) + \lambda \bbI$ is invertible, with inverse bounded in $L^{\infty}$.
Since $\rho$ is radially symmetric, the Fourier matrix associated to $\bbL$ is
\begin{equation*}
\bbM(\bfxi) = \intdm{\bbR^d}{\shapetensorby \big( 1-\cos(2 \pi \by \cdot \bfxi) \big) \rho(|\by|)}{\by}\,.
\end{equation*}
Let $R(\bfxi)$ be a rotation such that $\textstyle R(\bfxi)^{\intercal} \be_1 = \frac{\bfxi}{|\bfxi|}$. Then setting $\bh = |\bfxi| R(\bfxi)  \by$,
\begin{equation*}
\begin{split}
\bbM(\bfxi) &= \intdm{\bbR^d}{\frac{\rho \left( \frac{|\bh|}{|\bfxi|} \right)}{|\bfxi|^d} \frac{R(\bfxi)^{\intercal} \bh \otimes R(\bfxi)^{\intercal} \bh}{|\bh|^2} \big( 1-\cos(2 \pi h_1) \big)}{\bh} \\
	&= R(\bfxi)^{\intercal} \left( \intdm{\bbR^d}{\frac{\rho \left( \frac{|\bh|}{|\bfxi|} \right)}{|\bfxi|^d} \shapetensorbh \big( 1-\cos(2 \pi h_1) \big)}{\bh} \right) R(\bfxi)\,. \\
\end{split}
\end{equation*}
Using rotations, we see that the off-diagonal terms of $R \bbM R^{\intercal}$ are equal to $0$ for every $\bfxi \in \bbR^d$. Thus,
\begin{equation*}
\bbM(\bfxi) = R(\bfxi)^{\intercal} \diag(\ell_1(\bfxi), \ell_2(\bfxi), \ldots, \ell_d(\bfxi)) R(\bfxi)\,, \, 
\end{equation*}
where $\ell_i(\bfxi) = \intdm{\bbR^d}{\frac{\rho \left( \frac{|\bh|}{|\bfxi|} \right)}{|\bfxi|^d} \frac{h_i^2}{|\bh|^2} \big( 1-\cos(2 \pi h_1) \big)}{\bh}\,.$
Again using rotations, $\ell_2 = \ell_3 = \ldots = \ell_d$. Thus,
\begin{equation*}
\bbM(\bfxi) + \lambda \bbI = R(\bfxi)^{\intercal} \diag\big( \ell_1(\bfxi)+\lambda, \ell_2(\bfxi)+ \lambda, \ldots, \ell_2(\bfxi)+\lambda \big) R(\bfxi)\,.
\end{equation*}
Therefore,
\begin{equation*}
\big( \bbM(\bfxi) + \lambda \bbI \big)^{-1} = R(\bfxi)^{\intercal} \diag \left( \frac{1}{\ell_1(\bfxi)+\lambda}, \frac{1}{\ell_2(\bfxi)+ \lambda}, \ldots, \frac{1}{\ell_2(\bfxi)+\lambda} \right) R(\bfxi)\,.
\end{equation*}
Now, it is clear that $0 \leq \ell_i(\bfxi) \leq 2 \Vnorm{\rho}_{L^1(\bbR^d)}$ for almost every $\bfxi \in \bbR^d$. Thus, $\big( \bbM(\bfxi) + \lambda \bbI \big)^{-1} \in \big[ L^{\infty}(\bbR^d) \big]^{d \times d}$, so the linear operator $T : \big[ L^2(\bbR^d) \big]^d \to \big[ L^2(\bbR^d) \big]^d$ defined by $T \bff := \left( \big( \bbM(\bfxi) + \lambda \bbI \big)^{-1}  \widehat{\bff} \right)^{\vee}$ is bounded, with $\Vnorm{T \bff}_{L^2} \leq \lambda^{-1}\Vnorm{\bff}_{L^2}$. Further, $\bu := T\bff$ solves the equation \eqref{eq-GeneralKernel-Equation}.

For uniqueness, we show that any solution $\bu\in L^{2}$ of \eqref{eq-GeneralKernel-Equation} with data $\bff \equiv {\bf 0}$ is identically ${\bf 0}$. This follows from the fact that if $\bbL \bu + \lambda \bu = 0$, then  $\Vint{\bbL \bu + \lambda \bu, \bu}_{L^2(\bbR^d)} = 0$ and as proved earlier $\Vint{\bbL \bu, \bu}_{L^2(\bbR^d)}\geq 0$. 
\end{proof}

\begin{lemma}[Continuity of the operator: integrable case]\label{thm:GeneralIntegrableKernel:L2Continuity}
Suppose that $\rho$ is in {\bf Class A}.  Let $1 \leq p < \infty$. The operator $\bbL : \big[ L^p(\bbR^d) \big]^d \to \big[ L^p(\bbR^d) \big]^d$ is continuous. More precisely, we have for every $\bu \in \big[ L^p(\bbR^d) \big]^d$ the estimate
\begin{equation}\label{eq:IntegrableKernel:LpContinuity}
\Vnorm{\bbL \bu}_{L^p} \leq C \Vnorm{\bu}_{L^p}\,,
\end{equation}
where $C = C(\rho) > 0$.
\end{lemma}

\begin{proof}
Since $\rho \in L^1(\bbR^d)$,
\begin{equation*}
\begin{split}
\bbL \bu(\bx) &= \left( \intdm{\bbR^d}{\rho(\by) \shapetensorby}{\by} \right) \bu(\bx) - \intdm{\bbR^d}{\rho(\by) \shapetensorby \bu(\bx-\by)}{\by} \\
	&= \bbA \bu(\bx) + \bbK \ast \bu(\bx)\,,
\end{split}
\end{equation*}
where
\begin{equation*}
\bbA := \intdm{\bbR^d}{\rho(\by) \shapetensorby}{\by}\,, \qquad \bbK(\by) := \rho(\by) \shapetensorby\,.
\end{equation*}
Using the bound $|\bbA| \leq \Vnorm{\rho}_{L^1}$ and Young's inequality for integrals,
\begin{equation*}
\Vnorm{\bbL \bu}_{L^p} \leq \Vnorm{\bbA \bu}_{L^p} + \Vnorm{\bbK \ast \bu}_{L^p} \leq  2 \Vnorm{\rho}_{L^1} \Vnorm{\bu}_{L^p}\,.
\end{equation*}
\end{proof}

\begin{lemma}[\textit{A priori} estimate: integrable case] \label{thm-GeneralKernel-L2APrioriEstimate}
Suppose that $\rho$ is in {\bf Class A}.  Let $\lambda > 0$ be a constant and let $\bff \in \big[ L^2(\bbR^d) \big]^d$. Suppose $\bu \in \big[L^2(\bbR^d) \big]^d$ satisfies the equation \eqref{eq-GeneralKernel-Equation} in $\bbR^d$. Then there exists a constant $C =  C(\lambda) > 0$ such that
\begin{equation}\label{eq:AsymmetricL1Kernel:L2AprioriEstimate}
\Vnorm{\bu}_{L^2} \leq \lambda^{-1} \Vnorm{\bff}_{L^2}\,.
\end{equation}
\end{lemma}

\begin{proof}
Since $\big[ C^{\infty}_c(\bbR^d) \big]^d$ is dense in $\big[ L^2(\bbR^d) \big]^d$, by Lemma \ref{thm:GeneralIntegrableKernel:L2Continuity} it suffices to show \eqref{eq:AsymmetricL1Kernel:L2AprioriEstimate} for $\bu \in \big[ C^{\infty}_c(\bbR^d) \big]^d$.
Since
{\small \begin{align}\label{eq:SymmetricL1Kernel:L2AprioriEstimate:Proof1}
\intdm{\bbR^d}{|\bff|^2}{\bx} = \intdm{\bbR^d}{\left| \bbL \bu + \lambda \bu \right|^2}{\bx} = \intdm{\bbR^d}{\left| \bbL \bu \right|^2}{\bx} + 2 \lambda \intdm{\bbR^d}{\Vint{ \bbL \bu, \bu}}{\bx} + \lambda^2 \intdm{\bbR^d}{\left| \bu \right|^2}{\bx}\,,
\end{align}}
it suffices to show only that
\begin{equation}\label{eq:SymmetricL1Kernel:L2AprioriEstimate:Proof2}
\intdm{\bbR^d}{\Vint{ \bbL \bu, \bu}}{\bx} \geq 0\,.
\end{equation}
Then \eqref{eq:AsymmetricL1Kernel:L2AprioriEstimate} follows by dropping the first two terms on the right-hand side of \eqref{eq:SymmetricL1Kernel:L2AprioriEstimate:Proof1}.
To prove \eqref{eq:SymmetricL1Kernel:L2AprioriEstimate:Proof2}, note that
\begin{align*}
\intdm{\bbR^d}{\Vint{\bbL \bu, \bu}}{\bx} 
	= \intdm{\bbR^d}{\Vint{\widehat{\bbL \bu},\overline{\widehat{\bu}}}}{\bfxi} 
	&= \intdm{\bbR^d}{\Vint{\bbM(\bfxi) \widehat{\bu}(\bfxi), \overline{ \widehat{\bu}(\bfxi) } } }{\bfxi}\\
	&= \intdm{\bbR^d}{|\widehat{\bu}(\bfxi)|^2 \Vint{\bbM(\bfxi) \frac{\widehat{\bu}(\bfxi)}{|\widehat{\bu}(\bfxi)|} , \frac{\overline{\widehat{\bu}(\bfxi)}}{|\widehat{\bu}(\bfxi)|}} }{\bfxi}\,.
\end{align*}
Notice that $\Vint{\bbM(\bfxi)}$ is in general a matrix with complex entries. However, for every $\bv \in \bbC^d$ with $|\bv|=1$, a simple computation shows that 
\begin{equation*}
\Vint{\bbM(\bfxi)\bv,\overline{\bv}} = \intdm{\bbR^d}{\rho(\by) \big( 1 - \e^{2 \pi \imath \by \cdot \bfxi} \big) \left( \frac{(\Re \bv \cdot \by)^2 + (\Im \bv \cdot \by)^2}{|\by|^2} \right)}{\by}\,.
\end{equation*}
Since $\textstyle \intdm{\bbR^d}{\Vint{\bbL \bu, \bu}}{\bx}$ is real-valued, it follows that
\begin{align*}
\intdm{\bbR^d}{\Vint{\bbL \bu, \bu}}{\bx}  &= \intdm{\bbR^d}{|\widehat{\bu}(\bfxi)|^2 \Vint{\bbM(\bfxi) \frac{\widehat{\bu}(\bfxi)}{|\widehat{\bu}(\bfxi)|} , \frac{\overline{\widehat{\bu}(\bfxi)}}{|\widehat{\bu}(\bfxi)|}} }{\bfxi}\\
& = \intdm{\bbR^d}{|\widehat{\bu}(\bfxi)|^2 \Vint{\widetilde{\bbM}(\bfxi) \frac{\widehat{\bu}(\bfxi)}{|\widehat{\bu}(\bfxi)|} , \frac{\overline{\widehat{\bu}(\bfxi)}}{|\widehat{\bu}(\bfxi)|}} }{\bfxi}\,,
\end{align*}
where
\begin{equation*}
\widetilde{\bbM}(\bfxi) := \Re \bbM(\bfxi) = \intdm{\bbR^d}{\shapetensorby \big( 1-\cos(2 \pi \by \cdot \bfxi) \big) \rho(\by)}{\by}\,.
\end{equation*}
Clearly, $\Vint{ \widetilde{\bbM}(\bfxi) \bv, \overline{\bv} } \geq 0$ for all $|\bv| = 1$, and so
\begin{equation*}
\intdm{\bbR^d}{\Vint{\bbL \bu, \bu}}{\bx}  = \intdm{\bbR^d}{|\widehat{\bu}(\bfxi)|^2 \Vint{\widetilde{\bbM}(\bfxi) \frac{\widehat{\bu}(\bfxi)}{|\widehat{\bu}(\bfxi)|} , \frac{\overline{\widehat{\bu}(\bfxi)}}{|\widehat{\bu}(\bfxi)|}} }{\bfxi} \geq 0\,,
\end{equation*}
which is \eqref{eq:SymmetricL1Kernel:L2AprioriEstimate:Proof2}.
\end{proof}

The proof of Theorem \ref{thm:GeneralL1KernelAsymmetric:L2Solvability} for general $\rho$ in \textbf{Class A} now follows by applying the method of continuity, treating the operator $\bbL_t \bu := t (\bbL_0 \bu + \lambda \bu) + (1-t)( \bbL \bu + \lambda \bu)$. Here, $\bbL \bu $ denotes the operator with a general kernel $\rho$ belonging to \textbf{Class A}, and $\bbL_0 \bu$ denotes an operator with a kernel $\rho_0$ of the type considered in Theorem \ref{thm:Solvability:IntegrableKernel:Symmetric}, i.e. belonging to \textbf{Class A} and radially symmetric.
Thus the kernel of the operator $\bbL_t$ has the expression $t \rho(\by) + (1-t) \rho_0(|\by|)$, and satisfies the criteria of \textbf{Class A} with bounds independent of $t$. Therefore the \textit{a priori} estimates of Theorem \ref{thm-GeneralKernel-L2APrioriEstimate} can be leveraged for the operator $\bbL_t + \lambda$ in the method of continuity program, and the results of Theorem \ref{thm:GeneralL1KernelAsymmetric:L2Solvability} follow.

\section{$L^{p}$ solvability}
%
%
%
%
We use traditional Fourier multiplier techniques. The analogous result for the fractional Laplacian in the scalar case seems to be known, but we are unable to find a proof in the literature. We provide a proof of Theorem \ref{thm-FracLapKernel-LpSolvability} using the Marcinkiewicz multiplier theorem \cite{grafakos2008classical}.

Let $s \in (0,1)$. Consider the operator
\begin{equation*}
-\DotLap^s \bu(\bx) := \intdm{\bbR^d}{\frac{1}{|\by|^{d+2s}} \shapetensorby (\bu(\bx+\by)-\bu(\bx))}{\by}\,.
\end{equation*}
In the following calculation, we assume that $\bu \in \big[ \cS(\bbR^d) \big]^d$. 
Using the Fourier transform we can write $-\widehat{\DotLap^s \bu}$ as follows:
\begin{equation*}
-\widehat{\DotLap^s \bu}(\bfxi) = \left( \intdm{\bbR^d}{\frac{e^{2\pi \iota \bh \cdot \bfxi} - 1}{|\bh|^{d+2s}}  \shapetensorbh }{\bh} \right) \widehat{\bu}(\bfxi) := - \bbM^{\Delta}_s(\bfxi) \widehat{\bu}(\bfxi)\,.
\end{equation*}
From the proof of \cite[Theorem 4.2]{Mengesha-HalfSpace},
\begin{equation*}
\bbM^{\Delta}_s(\bfxi) = \big( 2 \pi |\bfxi| \big)^{2s} \left( \ell_1 \bbI + \ell_2 \shapetensorbfxi \right)\,,
\end{equation*}
where $\ell_1$ and $\ell_2$ are two positive constants depending only on $d$ and $s$.

Theorem \ref{thm-FracLapKernel-LpSolvability} hinges on the following lemma concerning Fourier multipliers.

\begin{lemma}\label{lma-LpContinuity}
For $s \in (0,1)$, $1 < p < \infty$ and $\lambda > 0$. Define the matrix of symbols
$$
\mathfrak{M}(\bfxi) := \big(1+4 \pi^2 |\bfxi|^2 \big)^{-s}\left( \bbM^{\Delta}_s(\bfxi) + \lambda \bbI \right). 
$$
Then both $\mathfrak{M}(\bfxi) $ and $\mathfrak{M}^{-1}(\bfxi)$  $L^p$-multipliers; that is, there exists a constant $C = C(d,s,p,\lambda)$ such that for every $\bu \in \big[ \cS(\bbR^d) \big]^d$
\begin{equation}
\Vnorm{\big( \frak{M} \widehat{\bu} \big)^{\vee} }_{L^p} \leq C \Vnorm{\bu}_{L^p}\,,\quad\text{and},\,\Vnorm{\big( \frak{M}^{-1} \widehat{\bu} \big)^{\vee} }_{L^p} \leq C \Vnorm{\bu}_{L^p}\,.
\end{equation}
\end{lemma}
We postpone the proof of the lemma for the moment, and prove Theorem \ref{thm-FracLapKernel-LpSolvability}.

\begin{proof}[Proof of Theorem \ref{thm-FracLapKernel-LpSolvability}]
First we show that the \textit{a priori} estimate \eqref{eq-FracLapKernel-LpAPrioriEstimate}
holds for every $\bu \in \big[ H^{2s,p}(\bbR^d) \big]^d$ solving \eqref{eq-FracLapKernel-Equation}.
By Lemma \ref{lma-LpContinuity} the operator $\DotLap^s + \lambda : \big[ H^{2s,p}(\bbR^d) \big]^d \to \big[ L^p(\bbR^d) \big]^d$ is continuous. Therefore, we need only show that \eqref{eq-FracLapKernel-LpAPrioriEstimate} holds for $\bu \in \big[ \cS(\bbR^d) \big]^d$. Then by definition and by Lemma \ref{lma-LpContinuity} we have that 
\begin{align*}
\Vnorm{\bu}_{H^{2s,p}(\bbR^d)} &= \Vnorm{\Big(1+4 \pi^s |\bfxi|^2 \big)^{s} \widehat{\bu} \Big)^{\vee}}_{L^p(\bbR^d)} \\
&= \Vnorm{\left( \frak{M}^{-1} \big( \bbM^{\Delta}_s + \lambda \bbI \big) \widehat{\bu} \right)^{\vee}}_{L^p(\bbR^d)} \\
&\leq C \Vnorm{\Big( \big( \bbM^{\Delta}_s + \lambda \bbI \big) \widehat{\bu} \Big)^{\vee}}_{L^p(\bbR^d)} = C \Vnorm{(-\mathring{\Delta})^s  \bu + \lambda \bu}_{L^p(\bbR^d)} = C \Vnorm{\bff}_{L^p(\bbR^d)}\,,
\end{align*}
which is \eqref{eq-FracLapKernel-LpAPrioriEstimate}. Uniqueness clearly follows from the a priori estimate as well. 

To see that a solution to \eqref{eq-FracLapKernel-Equation} exists, simply note that for every $\bff \in \big[ L^p(\bbR^d) \big]^d$ the distribution
\begin{equation*}
\bu(\bx) = \left( \left( \bbM^{\Delta}_s + \lambda \bbI \right)^{-1} \widehat{\bff} \right)^{\vee}(\bx)
\end{equation*}
belongs to $\big[ \cS'(\bbR^d) \big]^d$, and that
\begin{equation}
\left( (1 + 4 \pi^2 |\bfxi|^2)^s \widehat{\bu} \right)^{\vee} = \left( \mathfrak{M}^{-1} \widehat{\bff} \right)^{\vee} \in \big[ L^p(\bbR^d) \big]^d
\end{equation}
by Lemma \ref{lma-LpContinuity}. Thus $\bu$ is a function in $\big[ H^{2s,p}(\bbR^d) \big]^d$ that solves \eqref{eq-FracLapKernel-Equation}. The proof is complete.
\end{proof}


We use the Marcinkiewicz multiplier theorem to prove the lemma; see \cite{grafakos2008classical}.

\begin{proof}[Proof of Lemma \ref{lma-LpContinuity}]
We begin by showing that $\frak{M}(\bfxi)$ is an $L^{p}$ multiplier. 
By a direct computation, the function $\frak{M}(\bfxi)$ has the explicit expression
$$
\frak{M}(\bfxi) = \left(\frac{4 \pi^2 |\bfxi|^2}{1+ 4 \pi^2 |\bfxi|^2} \right)^s \left( \ell_1 \bbI + \ell_2 \shapetensorbfxi \right) + \frac{\lambda}{ \big( 1 + 4 \pi^2 |\bfxi|^2 \big)^s }\,.
$$
The symbol $\frac{\bfxi \otimes \bfxi}{|\bfxi|^2}$ is a matrix symbol whose $ij$-th entry is the symbol of the composition of Riesz transforms $-R_i R_j$. Thus the matrix symbol is an $L^p$-multiplier. The second expression in the sum defining $\frak{M}$ is the Bessel potential of order $-2s$, and is also an $L^p$-multiplier; see \cite[Chapter 5, Section 3.3]{stein}. Finally, it is established in \cite[Chapter V, Section 3.2., Lemma 2]{stein} that the symbol $\left( \frac{4 \pi^2 |\bfxi|^2}{1+ 4 \pi^2 |\bfxi|^2} \right)^s$ is a finite measure. Putting all this together gives us the result.

Next, we show that $\frak{M}^{-1}(\bfxi)$ an $L^p$ multiplier. 
Again by direct computation, the inverse $\frak{M}^{-1}(\bfxi)$ has the expression
\begin{equation*}
\frak{M}^{-1}(\bfxi) = \frac{\big( 1 + 4 \pi^2 |\bfxi|^2 \big)^s}{\ell_1 \big( 4 \pi^2 |\bfxi|^2  \big)^s + \lambda} \left( \bbI - \frac{\ell_2 \big( 4 \pi^2 |\bfxi|^2 \big)^s}{\big( 4 \pi^2 |\bfxi|^2 \big)^s (\ell_1 + \ell_2) + \lambda} \shapetensorbfxi \right)\,.
\end{equation*}
It suffices to show that the symbols
\begin{equation*}
m_1(\bfxi) := \frac{\big( 1 + 4 \pi^2 |\bfxi|^2 \big)^s}{\ell_1 \big( 4 \pi^2 |\bfxi|^2  \big)^s + \lambda}\,, \qquad m_2(\bfxi) :=  \frac{\ell_2 \big( 4 \pi^2 |\bfxi|^2 \big)^s}{\big( 4 \pi^2 |\bfxi|^2 \big)^s (\ell_1 + \ell_2) + \lambda}
\end{equation*}
are $L^p$-multipliers. Using this, we then observe that $\frak{M}^{-1}$ is a matrix whose entries consist of sums and products of $L^p$-multipliers, and the lemma is proved.

We use the Marcinkiewicz multiplier theorem; specifically, we show that $m_1$ and $m_2$ verify condition (6.2.9) in \cite[Corollary 6.2.5]{grafakos2008classical}. We introduce the auxiliary function $F : \bbR^{d+2} \to \bbR$ defined by
$$
F(\bfeta) := \frac{\big( |\eta_1|^2 + |\bfeta'|^2 \big)^s}{|\bfeta'|^{2s} + |\eta_2|^2}\,, \quad \bfeta = (\eta_1, \eta_2, \ldots, \eta_{d+1}, \eta_{d+2})\,, \quad \bfeta' := (\eta_3, \eta_4, \ldots, \eta_{d+1}, \eta_{d+2})\,.
$$
By inspection it is clear that $F$ is bounded on $\bbR^{d+2}$ and belongs to the class $C^{d+2}$ away from the coordinate axes on $\bbR^{d+2}$. Also by inspection, for any $\eta_1 \neq 0$ and $\eta_2 \neq 0$ the function $F(\eta_1, \eta_2, \cdot)$ is bounded on $\bbR^d$ and belongs to the class $C^d$ on $\bbR^d \setminus \{ \boldsymbol{0} \}$. 
The function $F$ is invariant under weighted dilation: for any $\mu > 0$, and $s>0$
\[
F(\mu\eta_1, \mu^{s}\eta_2, \mu\bfeta')= F(\eta_1, \eta_2, \bfeta'). 
\]
As a consequence, for any multi-index $\beta = (\beta_1, \beta_2, \ldots, \beta_{d+1}, \beta_{d+2})$ differentiation yields
\begin{equation}\label{weighted-dilation-derivative}
\mu^{|\beta|+(s-1)\beta_2} (\p_{\beta} F)(\mu \eta_1, \mu^{s} \eta_2, \mu \bfeta') = \p_{\beta} F(\eta_1, \eta_2, \bfeta')\,.
\end{equation}
Now for fixed $\eta_1 \neq 0$ and $\eta_2 \neq 0$ and every vector of the form $\bfeta = (\eta_1, \eta_2, \bfeta')\in \mathbb{R}^{d}$, we will choose $\mu=\mu_{\bfeta}$ such that $(\mu \eta_1, \mu^{s} \eta_2, \mu \bfeta')$ has unit length in $\mathbb{R}^{d +2}$. Such a $\mu> 0$  satisfies the equation
\[
\mu^{2} \eta_{1}^{2} + \mu^{2s}\eta_{2}^{2} + \mu^{2} |\bfeta'|^{2} = 1, 
\] and its existence can be shown as an intersection point of the quadratic function  $\mu\mapsto \mu^{2}(\eta_{1}^{2} +  |\bfeta'|^{2})$ and the curve $\mu \mapsto 1-\mu^{2s}\eta_{2}^{2}$.  Note that as $|\bfeta'| \to 0$, $\mu_{\bfeta}\to \mu_{0} > 0$, the intersection of the quadratic function  $\mu\mapsto \mu^{2}\eta_{1}^{2}$ and the curve $\mu \mapsto 1-\mu^{2s}\eta_{2}^{2}$. Also as  $|\bfeta'| \to \infty$, $\mu_{\bfeta}\to 1$. Moreover, for any $\bfeta' \neq 0$, $\mu_{\bfeta} < |\bfeta'|^{-1}$. As a consequence the set of unit vector 
\[
\mathbb{V} = \{ (\mu_{\bfeta} \eta_1, \mu_{\bfeta} ^{s} \eta_2, \mu_{\bfeta}  \bfeta') : \bfeta'\in \mathbb{R}^{d}\} \subset\mathbb{S}^{d+1}
\]
avoids all points of singularity of $F$ and its derivatives. 
 Now, if $\beta$ is a multi-index with $ \beta_1 = 0 = \beta_2$, then by construction $\mu_{\bfeta}^{\beta_j} \leq |\eta_j|^{-\beta_j}$ for $j \in \{ 3, 4, \ldots, d+1, d+2 \}$. 
Plugging this $\mu_{\bfeta}$ in \eqref{weighted-dilation-derivative}, we have that for any $\bfeta' \neq 0,$
$$
\left| \p_{\beta} F(\eta_1,\eta_2,\bfeta')  \right| \leq \sup_{\bv \in \mathbb{V}} |\p_{\beta} F(\bv)| \mu^{\beta_3 + \beta_4 + \ldots + \beta_{d+2}} \leq C |\eta_3|^{-\beta_3} \cdots |\eta_{d+2}|^{-\beta_{d+2}}\,
$$
for some constant $C$ that depends on $\eta_1$, $\eta_2$, $\beta$ and $s$.  
Therefore, for any multi-index $\gamma = (\gamma_1, \ldots, \gamma_d)$, $\gamma_{i} \in \{1, 2, \cdots, d\}, $ we set $\beta = (0,0,\gamma_1, \ldots, \gamma_d)$ and obtain
\begin{equation*}
|\p_{\gamma} m_1 (\bfxi)| = \frac{1}{\ell_1} \left|\p_{\beta} F \Big(1 , \sqrt{\lambda \ell_1^{-1}},  2 \pi |\bfxi|  \Big) \right| \leq C |\xi_1|^{-\gamma_1} \cdots |\xi_d|^{-\gamma_d}\,.
\end{equation*}
Thus the hypotheses of \cite[Corollary 6.2.5]{grafakos2008classical} are satisfied for $m_1(\bfxi)$.

A similar strategy fails for $m_2(\bfxi)$ because the numerator fails to be bounded away from zero. Specifically, $m_2(\bfxi)$ is $C^d$ on $\bbR^d$ away from the coordinate axes. We must check the derivatives directly. Introduce the auxiliary function $G_a : (0,\infty) \to \bbR$ defined by
\begin{equation*}
G_a(\eta) := \frac{\eta^s}{\eta^s + a} = \frac{1}{1+ a \eta^{-s}} \,, \qquad a > 0\,.
\end{equation*}
Then it is clear that for every $n \in \bbN$ the function $G_a$ satisfies
\begin{equation*}
\left| G_a^{(n)}(\eta) \right| \leq C \eta^{-n}\,.
\end{equation*}
Let $a = \frac{\lambda}{(\ell_1 + \ell_2) 4 \pi^2}$, and let $\beta$ be a multi-index with $\beta_j \in \{ 0,1 \}$. Then
$$
\p_{\beta} m_2(\bfxi) = \frac{\ell_2}{\ell_1 + \ell_2}  G_a^{(|\beta|)} \big( |\bfxi|^2 \big)  \, \, 2^{|\beta|} \prod_{j=1}^d \xi_j^{\beta_j}
$$
and so
$$
\big| \p_{\beta} m_2(\bfxi) \big| \leq C |\bfxi|^{-2 |\beta|} \prod_{j=1}^d |\xi_j|^{\beta_j} \leq C \prod_{j=1}^d |\xi_j|^{-\beta_j}\,.
$$

Thus $m_2$ also satisfies the hypotheses of of \cite[Corollary 6.2.5]{grafakos2008classical}, and the proof is complete.
\end{proof}

\begin{remark}
It is also possible to prove that $m_2(\bfxi)$ is an $L^p$-multiplier without using the Marcinkiewicz multiplier theorem. In fact, one can show using the proof of \cite[Chapter 5, Section 3.2, Lemma 2]{stein} that $m_2(\bfxi)$ is in fact the Fourier transform of a finite measure, hence an $L^p$-multiplier.
\end{remark}

\section{The Peridynamic-Type Wave Equation}\label{sec:TimeDepProb}
In this section we turn our attention to the system of time dependent equations given in \eqref{eq:WaveEquation-MR} and prove Theorem \ref{thm:WaveEqnSolvability}. We recall that we are interested in the system 
\begin{equation}\label{eq:WaveEquation} 
\begin{cases}
\p_{tt} \bu + \bbL \bu  + \lambda \bu = \bff\,, &\qquad  (\bx,t) \text{ in } \bbR^d \times [0,T]\,, \\
\bu(\bx,0) = \bu_0(\bx)\,, &\qquad \bx \text{ in } \bbR^d\,, \\
\p_t \bu(\bx,0) = \bv_0(\bx)\,, &\qquad \bx \text{ in } \bbR^d\,,
\end{cases}
\end{equation}
where $\bff \in C^1 \left( [0,T] ; \big[ L^2(\bbR^d) \big]^d \right)$, $\bu_0 \in \big[H^{2s,2}(\bbR^d) \big]^d$, and $\bv_0 \in \big[ H^{s,2}(\bbR^d) \big]^d$.

We use the semigroup approach described in \cite[Chapter 10]{Brezis-Book}. To this end, we assume for the rest of this section that $\rho$ is in {\bf Class B} and that $m(\by)$ is an even function. In this case $\bbL$ becomes
\begin{equation}
-\bbL \bu(\bx) = \pv \intdm{\bbR^d}{ \shapetensorby \big( \bu(\bx+\by)-\bu(\bx) \big)  \rho(\by) }{\by}\,.
\end{equation}

Let $\lambda_0$ be the quantity defined in Theorem \ref{thm-GeneralKernel-L2Solvability}, and for $\lambda > \lambda_0 - 1$ define the operator
\begin{equation}
\bbL_{\lambda} := \bbL + \lambda \bbI : \big[H^{2s,2}(\bbR^d) \big]^d \to \big[ L^2(\bbR^d) \big]^d\,.
\end{equation}
Note that by the symmetry assumption on $\bbL$, 
\begin{equation*}
\begin{split}
\Vint{\bbL \bu, \bu}_{L^2(\bbR^d)} 
&= \iintdm{\bbR^d}{\bbR^d}{\rho(\bx-\by) \diffqbunorm^2 }{\by}{\bx} 
\geq 0
\end{split}
\end{equation*}
for every $\bu \in \big[H^{2s,2}(\bbR^d) \big]^d$. Therefore for any $\lambda \geq 0$ we have 
\begin{equation}\label{eq:LMonotone}
\Vint{\bbL \bu + \lambda \bu, \bu}_{L^2(\bbR^d)} \geq 0\,\quad \forall \bu \in \big[H^{2s,2}(\bbR^d) \big]^d\,.
\end{equation}
Further, Theorem \ref{thm-GeneralKernel-L2Solvability} implies that: 
\begin{equation}\label{eq:LMaximal}
\begin{split}
\text{ for every } \bg \in \big[ L^2(\bbR^d) \big]^d & \text{ there exists a unique } \bu \in \big[ H^{2s,2}(\bbR^d) \big]^d\\
 \text{ satisfying } 
	&\bbL_{\lambda} \bu + \bu = \bg \text{ in } \bbR^d\,,
\end{split}
\end{equation}
since $\lambda > \lambda_0 -1 $. Equations \eqref{eq:LMonotone} and \eqref{eq:LMaximal} define $\mathbb{L}_{\lambda}$ as a maximal monotone operator for any nonnegative $\lambda$ satisfying $\lambda > \lambda_0 - 1$.  
Since the unbounded operator $\bbL_{\lambda}$ is symmetric, by \eqref{eq:LMonotone} and \eqref{eq:LMaximal} it is self-adjoint on $\big[ H^{2s,2}(\bbR^d) \big]^d$ as well, see \cite[Proposition 7.6]{Brezis-Book}.
Therefore the positive square root of $\bbL_{\lambda}$, denoted $\bbL_{\lambda}^{1/2}$, is well-defined. Therefore, by the estimate in Theorem \ref{thm-GeneralKernel-L2Solvability} we have that for every $\bu \in \big[ H^{2s,2}(\bbR^d) \big]^d$ the norm
\begin{equation}\label{eq:SquareRootProperties}
\begin{split}
\Vint{\bbL_{\lambda}^{1/2} \bu, \bbL_{\lambda}^{1/2} \bu}_{L^2(\bbR^d)} &= \Vint{\bbL \bu, \bu}_{L^2(\bbR^d)} + \lambda \Vnorm{\bu}_{L^2}^2 
\end{split}
\end{equation}
is equivalent to the norm $\Vnorm{(-\Delta)^{s/2} \bu}_{L^2(\bbR^d)}^2 + \Vnorm{\bu}_{L^2(\bbR^d)}^2\,.$
Thus $\Vint{\bbL_{\lambda}^{1/2} \big( \cdot \big), \bbL_{\lambda}^{1/2} \big( \cdot \big)}$ defines an equivalent inner product on $\big[ H^{s,2}(\bbR^d) \big]^d$.

Now, we rewrite the system \eqref{eq:WaveEquation} as a larger system of equations
\begin{equation}\label{eq:WaveEquationSystemExpanded}
\begin{cases}
\bv = \p_t \bu\,, &\qquad (\bx,t) \text{ in } \bbR^d \times [0,T]\,, \\
\p_t \bv + \bbL \bu  + \lambda \bu = \bff\,, &\qquad  (\bx,t) \text{ in } \bbR^d \times [0,T]\,, \\
\bu(\bx,0) = \bu_0(\bx)\,, &\qquad \bx \text{ in } \bbR^d\,, \\
\p_t \bu(\bx,0) = \bv_0(\bx)\,, &\qquad \bx \text{ in } \bbR^d\,,
\end{cases}
\end{equation}
or equivalently
\begin{equation}\label{eq:WaveEquationSystem}
\begin{cases}
\p_t \bU + \frak{A}_{\lambda} \bU = \bF\,, &\qquad  (\bx,t) \text{ in } \bbR^d \times [0,T]\,, \\
\bU(\bx,0) = \bU_0(\bx)\,, &\qquad \bx \text{ in } \bbR^d\,, \\
\end{cases}
\end{equation}
where $\bU = (\bu,\bv)^{\intercal}$, $\bF = ({\bf 0},\bff)^{\intercal}$, $\bU_0 = (\bu_0, \bv_0)^{\intercal}$ and the operator $\frak{A}_{\lambda}$ is defined as
\begin{equation*}
\frak{A}_{\lambda} := 
	\begin{pmatrix}
	0 & - \bbI \\
	\bbL_{\lambda} & 0
	\end{pmatrix}\,.
\end{equation*}
We denote the Hilbert space $\cH := \big[ H^{s,2}(\bbR^d) \big]^d \times \big[ L^2(\bbR^d) \big]^d$, and we define the domain of the operator $\mathfrak{A}_{\lambda}$ by $D(\frak{A}_{\lambda}) := \big[H^{2s,2}(\bbR^d) \big]^d \times \big[ H^{s,2}(\bbR^d) \big]^d \subset \cH$. 
By \eqref{eq:SquareRootProperties}, the inner product on $\cH$ can be defined by
\begin{equation}\label{eq:InnerProduct}
\Vint{\bU, \bV}_{\cH} := \intdm{\bbR^d}{ \Vint{\bbL_{\lambda}^{1/2} \bu_1, \bbL_{\lambda}^{1/2} \bv_1} }{\bx} + \intdm{\bbR^d}{\Vint{\bu_1, \bv_1}}{\bx} + \intdm{\bbR^d}{\Vint{\bu_2, \bv_2}}{\bx}
\end{equation}
where $\bU = (\bu_1, \bu_2)$, $\bV = (\bv_1, \bv_2)$. 

Denoting the identity operator on $\cH$ by $\frak{I}$, we note that the operator  $\frak{A}_{\lambda} + \frak{I}$ is nonnegative. Indeed, for any $\bU = (\bu,\bv)^{\intercal} \in D(\frak{A}_{\lambda})$ we have using \eqref{eq:InnerProduct} that
\begin{equation}\label{eq:AMonotone}
\begin{split}
&\Vint{(\frak{A}_{\lambda} + \frak{I} )\bU, \bU}_{\cH} \\
& = \Vint{\frak{A}_{\lambda} \bU, \bU}_{\cH} + \Vint{\bU, \bU}_{\cH} \\
	&= - \intdm{\bbR^d}{ \Vint{\bbL_{\lambda}^{1/2} \bu, \bbL_{\lambda}^{1/2} \bv} }{\bx} - \intdm{\bbR^d}{\Vint{\bu, \bv}}{\bx} + \intdm{\bbR^d}{\Vint{\bbL_{\lambda} \bu, \bv}}{\bx} \\
	&\quad + \intdm{\bbR^d}{ |\bbL_{\lambda}^{1/2} \bu|^2 }{\bx} + \intdm{\bbR^d}{|\bu|^2}{\bx} + \intdm{\bbR^d}{|\bv|^2}{\bx} \\
	&= - \intdm{\bbR^d}{ \Vint{\bbL_{\lambda} \bu,  \bv} }{\bx} + \intdm{\bbR^d}{\Vint{\bbL_{\lambda} \bu, \bv}}{\bx} + \intdm{\bbR^d}{ |\bbL_{\lambda}^{1/2} \bu|^2 }{\bx} \\
	&\quad  + \intdm{\bbR^d}{|\bu|^2}{\bx} - \intdm{\bbR^d}{\Vint{\bu, \bv}}{\bx} + \intdm{\bbR^d}{|\bv|^2}{\bx} \\
	&= \intdm{\bbR^d}{ |\bbL_{\lambda}^{1/2} \bu|^2 }{\bx} + \frac{1}{2} \intdm{\bbR^d}{(|\bu|^2 + |\bv|^2)}{\bx} + \frac{1}{2} \intdm{\bbR^d}{|\bu - \bv|^2}{\bx} \geq 0\,.
\end{split}
\end{equation}
In addition, the operator $\frak{A}_{\lambda} + 2 \frak{I}$ is invertible. To see this, note that  inverting $\frak{A}_{\lambda}$ is equivalent to solving the system
\begin{equation}\label{eq:AMaximalSystem}
\begin{cases}
- \bv + 2 \bu = \bg_1\,, \quad
\bbL \bu + \lambda \bu + 2 \bv = \bg_2\,, &\quad \bx \in \bbR^d\,,
\end{cases}
\end{equation}
for any $(\bg_1, \bg_2)^{\intercal} \in \cH$, which is equivalent to solving
\begin{equation}\label{eq:AMaximal}
\bbL \bu + (\lambda + 4) \bu = 2 \bg_1 + \bg_2\,, \quad \bx \in \bbR^d\,.
\end{equation}
Since $\lambda + 4 > \lambda_0$ and since $2 \bg_1 + \bg_2 \in \big[ L^2(\bbR^d) \big]^d$, \eqref{eq:AMaximal} has a unique solution $\bu \in \big[ H^{2s,2}(\bbR^d) \big]^d$. Then $\bv = 2 \bu - \bg_1 \in \big[ H^{s,2}(\bbR^d) \big]^d$ and so $\bU = (\bu, \bv)^{\intercal} \in D(\frak{A}_{\lambda})$ solves \eqref{eq:AMaximalSystem}.

Therefore, since $\frak{A}_{\lambda} + \frak{I}$ is \textit{maximal monotone} for $\lambda > \lambda_0 - 1$ nonnegative, by the Hille-Yosida theorem \cite[Theorem 1, Chapter XVII, section 3]{Dautray-Lions}, there is a continuous contraction semigroup $S_{{\frak{A}_{\lambda}}+ \mathfrak{I}}(t):L^{2}\to L^2$ such that for any $\bG \in [L^{2}(\mathbb{R}^{d})]^{2d}$ the solution  $\bU$ of 
\begin{equation}\label{eq:WaveEquationSystemPlusId}
\begin{cases}
\p_t \bU + \frak{A}_{\lambda} \bU + \mathfrak{I} \bu = \bG\,, &\qquad  (\bx,t) \text{ in } \bbR^d \times [0,T]\,, \\
\bU(\bx,0) = \bU_0(\bx)\,, &\qquad \bx \text{ in } \bbR^d\,, \\
\end{cases}
\end{equation}
can be given by the formula \cite[Theorem 7.10]{Brezis-Book}
\begin{equation*}
\bU(\bx,t) = S_{{\frak{A}_{\lambda}}+\mathfrak{I}}(t) \bU_0(\bx) + \intdmt{0}{t}{S_{{\frak{A}_{\lambda}}+\mathfrak{I}}(t-s) \bG(\bx,s)}{s}\,.
\end{equation*}
Taking  $\bG = \rme^{t}\bF$ and making the substitution 
$\bW = \rme^t \bU$, we see that $\bW$ is the unique solution to \eqref{eq:WaveEquationSystem} with source data $\bF$. Thus the first component of $\bW$ is the unique solution to \eqref{eq:WaveEquation}.

The conservation law 
\eqref{eq:ConservationLaw-MR} follows from multiplying the equation by $\p_t \bu$ and integrating by parts.

%
%
%

%

\bibliography{References}
\bibliographystyle{plain}

\end{document}